\documentclass[oneside,english]{amsart}
\usepackage[T1]{fontenc}
\usepackage[latin9]{inputenc}
\usepackage{amstext}
\usepackage{amsthm}
\usepackage{amssymb}

\makeatletter
\numberwithin{equation}{section}
\numberwithin{figure}{section}
\theoremstyle{plain}
\newtheorem{thm}{\protect\theoremname}[section]
\theoremstyle{plain}
\newtheorem{conjecture}[thm]{\protect\conjecturename}
\theoremstyle{remark}
\newtheorem{rem}[thm]{\protect\remarkname}
\theoremstyle{plain}
\newtheorem{lem}[thm]{\protect\lemmaname}
\theoremstyle{plain}
\newtheorem{cor}[thm]{\protect\corollaryname}
\theoremstyle{plain}
\newtheorem{prop}[thm]{\protect\propositionname}

\usepackage[top=30truemm,bottom=30truemm,left=25truemm,right=25truemm]{geometry}
\usepackage{graphics}

\newcommand{\stodd}{\reflectbox{$\ddots$}}

\allowdisplaybreaks[4]

\author{Nozomi Ito}
\keywords{Miyawaki lifts, refined GGP conjecture}
\subjclass{Primary 11F67; Secondary 22E50}

\makeatother

\usepackage{babel}
\providecommand{\conjecturename}{Conjecture}
\providecommand{\corollaryname}{Corollary}
\providecommand{\lemmaname}{Lemma}
\providecommand{\propositionname}{Proposition}
\providecommand{\remarkname}{Remark}
\providecommand{\theoremname}{Theorem}

\begin{document}
\global\long\def\ad{\text{ad}}%
\global\long\def\abs{|\cdot|}%
\global\long\def\irr{\mathrm{Irr}}%
\global\long\def\spec{\mathrm{Spec}}%
\global\long\def\gll{\mathrm{GL}}%
\global\long\def\nat{\mathbb{\mathbb{Z}}_{>0}}%
\global\long\def\defi{\overset{\mathrm{def}}{\iff}}%
\global\long\def\map{\rightarrow}%
\global\long\def\ep{\varepsilon}%
\global\long\def\sll{\mathrm{SL}}%
\global\long\def\spp{\mathrm{Sp}}%
\global\long\def\diag{\mathrm{diag}}%
\global\long\def\tr{\mathrm{Tr}}%
\global\long\def\akk{\mathbb{A}_{k}}%
\global\long\def\aff{\mathbb{A}_{F}}%
\global\long\def\zz{\mathbb{Z}}%
\global\long\def\rr{\mathbb{R}}%
\global\long\def\qq{\mathbb{Q}}%
\global\long\def\cc{\mathbb{C}}%
\global\long\def\inj{\hookrightarrow}%
\global\long\def\surj{\twoheadrightarrow}%
\global\long\def\oo{\infty}%
\global\long\def\bu{\mathrm{U}}%
\global\long\def\sgn{\mathrm{\mathrm{sgn}}}%
\global\long\def\ind{\mathrm{\mathrm{Ind}}}%
\global\long\def\vol{\mathrm{\mathrm{vol}}}%
\global\long\def\aqq{\mathbb{A}_{\qq}}%
\global\long\def\nami{\rightsquigarrow}%
\global\long\def\lig{\mathfrak{g}}%
\global\long\def\lik{\mathfrak{k}}%
\global\long\def\lip{\mathfrak{p}}%
\global\long\def\lit{\mathfrak{t}}%
\global\long\def\liu{\mathfrak{u}}%
\global\long\def\lih{\mathfrak{h}}%
\global\long\def\rint{\mathcal{O}}%

\title{On the formula for the norms of Miyawaki lifts}
\begin{abstract}
The Miyawaki lifting is a lifting of Siegel modular forms introduced
by Ikeda in his 2006 paper. In the same paper, he also conjectured
a formula for the norms of Miyawaki lifts. In this paper, we show
that his conjectural formula can be rewritten into a refined GGP type
formula.
\end{abstract}

\maketitle

\section{Introduction}

Let $\mathbf{f}\in S_{2k}(\sll_{2}(\zz))$ be a normalized Hecke eigenform,
where $k\in\zz_{>0}$. Denote the Ikeda lift (see \cite{MR1884618})
of $\mathbf{f}$ in $S_{k+n+r}(\spp_{4n+4r}(\zz))$ by $\mathbf{F}$,
where $n\in\zz_{\geq0}$ and $r\in\zz_{>0}$ such that $k+n+r\in2\zz$.
For a Hecke eigenform $\mathbf{g}\in S_{k+n+r}(\spp_{2r}(\zz)),$
define $\mathbf{F_{g}}\in S_{k+n+r}(\spp_{4n+2r}(\zz))$ by

\[
\mathbf{F_{g}}(Z)=\langle\mathbf{F}|_{\lih_{2n+r}\times\lih_{r}}(Z,-),\mathbf{g}\rangle
\]
for $Z\in\lih_{2n+r}$, where $\lih_{2n+r}\times\lih_{r}=\diag(\lih_{2n+r},\lih_{r})$
and $\langle-,-\rangle$ denotes the Petersson inner product. This
cusp form is called the Miyawaki lift of $\mathbf{g}$ with respect
to $\mathbf{F}$, which was introduced by Ikeda in \cite{MR2219248}.
It has the following remarkable property.
\begin{thm}[{\cite[Theorem 1.1]{MR2219248}}]
If $\mathbf{F_{g}}\not\equiv0,$ then it is a Hecke eigenform whose
(non-complete) standard L-function $L(s,\mathbf{F_{g}},\mathrm{st})$
satisfies

\[
L(s,\mathbf{F_{g}},\mathrm{st})=L(s,\mathbf{g},\mathrm{st})\prod_{i=1}^{2n}L(s+k+n-i,\mathbf{f}).
\]
\end{thm}

However, the assumption in this theorem is an obstacle to applying
Miyawaki lifts. There are several results on the nonvanishing of $\mathbf{F_{g}}$
(cf. \cite{MR4081121}, \cite{MR4034599}), but we have not yet been
able to determine when $\mathbf{F_{g}}\not\equiv0$.

The most straightforward way to know when $\mathbf{F_{g}}\not\equiv0$
is to consider the norm of $\mathbf{F_{g}}$. In \cite{MR2219248},
Ikeda studied some examples and gave a conjectural formula for the
norms of Miyawaki lifts.
\begin{conjecture}[{\cite[Conjecture 5.1]{MR2219248}}]
\label{conj:Assume-that-.}We have
\begin{equation}
\frac{\langle\mathbf{f},\mathbf{f}\rangle\langle\mathbf{\mathbf{F_{g}}},\mathbf{\mathbf{F_{g}}}\rangle}{\langle\mathbf{h},\mathbf{h}\rangle\langle\mathbf{g},\mathbf{g}\rangle}=2^{-(a+c)}\Lambda(k+n,\mathrm{st}(g)\boxtimes f)\prod_{i=1}^{n}\Lambda(2i-1,f,\mathrm{Ad})\xi(2i)\Gamma_{\rr}(2i-1)\Gamma_{\rr}(2i+1),\label{eq:ikeconj}
\end{equation}
where $\mathbf{h}\in S_{k+1/2}^{+}(\Gamma_{0}(4))$ is the Hecke eigenform
corresponds to $\mathbf{f}$ by the Shimura correspondence, $\Lambda(s,-)$
means the completion of $L(s,-)$, $\xi(s)$ is the completed zeta
function, and
\[
a=r^{2}+2k(n+r)+2rn+2n+r-k-2,\ c=\begin{cases}
0 & n>0,\\
1 & n=0.
\end{cases}
\]
\end{conjecture}

We note that Ichino solved this conjecture for $(r,n)=(2,0)$ in \cite{MR2198222}.
\begin{rem}
The original version of Conjecture \ref{conj:Assume-that-.} assumes
that $n<k$ based on Deligne\textquoteright s conjecture (see \cite[Remark 5.2]{MR2219248}).
In addition, it also gives a formula for $r=0$. The formula for $r=0$
was shown by Katsurada and Kawamura without assuming $n<k$ (see Theorem
\ref{thm:-We-havewhere}).
\end{rem}

The aim of this paper is to rewrite the above conjecture using L-values
and local integrals, like the refined Gan-Gross-Prasad conjecture
\cite{MR2585578}. Let $\sigma=\otimes_{v}\sigma_{v},\ \Sigma=\otimes_{v}\Sigma_{v},\ \pi=\otimes_{v}\pi_{v}$
be the cusp representations of $\gll_{2}(\aqq)$, $\spp_{4n+4r}(\aqq)$,
$\spp_{2r}(\aqq)$ correspond to $\mathbf{f},\ \mathbf{F},\ \mathbf{g},$
respectively. Let $\varphi=\otimes_{v}\varphi_{v}\in\Sigma$ and $\psi=\otimes_{v}\psi_{v}\in\pi$
be the adelizations of $\mathbf{F}$ and $\mathbf{g}$, respectively.
Put 
\[
\varphi_{\psi}(h)=(\varphi|_{\spp_{4n+2r}\times\spp_{2r}(\aqq)}(h,-),\psi)
\]
for $h\in\spp_{4n+2r}(\aqq)$, where $\spp_{4n+2r}\times\spp_{2r}$
is identified with 
\[
\left\{ \begin{pmatrix}A &  & B\\
 & E &  & F\\
C &  & D\\
 & G &  & H
\end{pmatrix}\ \middle|\ \begin{pmatrix}A & B\\
C & D
\end{pmatrix}\in\spp_{4n+2r},\ \begin{pmatrix}E & F\\
G & H
\end{pmatrix}\in\spp_{2r}\right\} \subset\spp_{4n+4r}
\]
 and we denote by $(-,-)$ the $L^{2}$-inner product with respect
to the Tamagawa measure. For each place $v$ of $\qq,$ put 
\[
\Phi_{v}(h_{v})=(\Sigma_{v}(h_{v})\varphi_{v},\varphi_{v})/(\varphi_{v},\varphi_{v}),\ \Psi_{v}(g_{v})=(\pi_{v}(g_{v})\psi_{v},\psi_{v})/(\psi_{v},\psi_{v})
\]
for $h_{v}\in\spp_{4n+4r}(\qq_{v}),\ g_{v}\in\spp_{2r}(\qq_{v})$
and define $I(\Phi_{v},\Psi_{v})$ by

\[
I(\Phi_{v},\Psi_{v})=\int_{\spp_{2r}(\qq_{v})}\Phi_{v}(g_{v})\overline{\Psi_{v}(g_{v})}dg_{v}
\]
as long as it converges, where $(-,-)$ denotes a fixed inner product
and $dg_{v}$ is the Haar measure on $\spp_{2r}(\qq_{v})$ such that
\[
\mathrm{vol}(\spp_{2r}(\zz)\backslash\spp_{2r}(\rr),dg_{\infty})=\zeta(2)\zeta(4)\dots\zeta(2r),\ \mathrm{vol}(\spp_{2r}(\zz_{p}),dg_{p})=\zeta_{p}(2)^{-1}\zeta_{p}(4)^{-1}\dots\zeta_{p}(2r)^{-1}\ (p<\infty).
\]
Note that $\prod_{v}dg_{v}$ is the Tamagawa measure on $\spp_{2r}(\aqq)$.
We will see later that $I(\Phi_{v},\Psi_{v})$ is well-defined if
$\pi$ is tempered (Corollary \ref{cor:If--is}).

Then, the rewritten conjecture is the following. 
\begin{conjecture}
\label{conj:Assume-that-}Assume that $\pi_{p}$ is tempered for each
finite place $p$ of $\qq$. Then, for any finite set $S$ of places
of $\qq$ containing $\infty$, we have

\begin{equation}
\frac{(\varphi_{\psi},\varphi_{\psi})}{(\varphi,\varphi)(\psi,\mathbf{\psi})}=2^{-c}\mathcal{L}^{S}\prod_{v\in S}I(\Phi_{v},\Psi_{v}),\label{eq:conjggp}
\end{equation}
where $\mathcal{L}^{S}$ is an explicit partial Euler-product (see
$\S$\ref{sec:Main-theorem}).
\end{conjecture}

This paper is organized as follows. In $\S$\ref{sec:Deformation-of-the},
we deform (\ref{eq:ikeconj}) into an equation between the left hand
side of (\ref{eq:conjggp}) and some L-values. In $\S$\ref{sec:Main-theorem}
we state the main theorem of this paper. The main theorem gives the
value of $I(\Phi_{v},\Psi_{v})$ and establishes the equivalence between
Conjecture \ref{conj:Assume-that-.} and Conjecture \ref{conj:Assume-that-}.
In $\S$\ref{sec:Proof-of-theorem} we prove the main theorem. The
proof is quite different whether $v=\infty$ or $v=p<\infty$. In
the former case, we use a single fact about matrix coefficients of
lowest weight modules of scalar type. In the latter case, we use (local)
double descent of Ginzburg-Soudry \cite{zbMATH07485546}. I believe
that the latter method is essential to prove Conjecture \ref{conj:Assume-that-},
but I have not found a way to apply it to the former case.

\section*{Notations}

For any ring $R$, define $\spp_{2d}(R)$ by

\[
\spp_{2d}(R)=\left\{ g\in\gll_{2d}(R)\ \middle|\ g\begin{pmatrix} & -1_{d}\\
1_{d}
\end{pmatrix}{}^{t}g=\begin{pmatrix} & -1_{d}\\
1_{d}
\end{pmatrix}\right\} .
\]
For $d=a+b$, we identify $\spp_{2a}\times\spp_{2b}$ with the subgroup
\[
\left\{ \begin{pmatrix}A &  & B\\
 & E &  & F\\
C &  & D\\
 & G &  & H
\end{pmatrix}\ \middle|\ \begin{pmatrix}A & B\\
C & D
\end{pmatrix}\in\spp_{2a},\begin{pmatrix}E & F\\
G & H
\end{pmatrix}\in\spp_{2b}\right\} 
\]
of $\spp_{2d}$. We often denote the first factor of it by $\spp_{2a}$
and the second one by $\spp_{2b}$. (Strictly speaking, this notation
is useless when $a=b$, but this will not be a problem in this paper.) 

We denote by $S_{w}(\spp_{2d}(\zz))$ the space of Siegel cusp forms
of degree $d$, weight $w$, and level one. We also denote by $\lih_{d}$
the Siegel upper half space of degree $d$. For any $S_{w}(\spp_{2d}(\zz))$,
$\langle-,-\rangle$ means the Petersson inner product on $S_{w}(\spp_{2d}(\zz))$.

Let $\tau$ be a unitary representation of a group $G$. Then, we
denote a fixed $G$-invariant inner product of $\tau$ by $(-,-)_{\tau}$
or simply $(-,-)$. Moreover, if $G=\spp_{2d}(\aqq)$ for some $d\in\zz_{>0}$
and $\tau$ is a subrepresentation of the space of square-integrable
automorphic forms on $G,$ then we fix $(-,-)$ as the $L^{2}$-inner
product with respect to the Tamagawa measure.

\section{Deformation of the formula\label{sec:Deformation-of-the}}

In the rest of this paper, we continue the notations of the previous
section. Namely,
\begin{itemize}
\item $k\in\zz_{>0},$ $n\in\zz_{\geq0}$, and $r\in\zz_{>0}$ such that
$k+n+r\in2\zz$,
\item $\mathbf{f}\in S_{2k}(\sll_{2}(\zz))$ is a normalized Hecke eigenform
and $\tau=\otimes_{v}\tau_{v}$ is the cusp representation of $\gll_{2}(\aqq)$
corresponds to $\mathbf{f}$,
\item $\mathbf{F}$ is the Ikeda lift of $\mathbf{f}$ in $S_{k+n+r}(\spp_{4n+4r}(\zz))$
and $\varphi=\otimes_{v}\varphi_{v}$ is the adelization of $\mathbf{F}$
and $\Sigma=\otimes_{v}\Sigma_{v}$ is the cusp representation of
$\spp_{4n+4r}(\aqq)$ generated by $\varphi$,
\item $\mathbf{g}\in S_{k+n+r}(\spp_{2r}(\zz))$ is a Hecke eigenform and
$\psi=\otimes_{v}\psi_{v}$ is the adelization of $\mathbf{g}$ and
$\pi=\otimes_{v}\pi_{v}$ is the cusp representation of $\spp_{2r}(\aqq)$
generated by $\psi$,
\item $\Phi_{v}$ and $\Psi_{v}$ are defied by
\[
\Phi_{v}(h_{v})=(\Sigma_{v}(h_{v})\varphi_{v},\varphi_{v})/(\varphi_{v},\varphi_{v}),\ \Psi_{v}(g_{v})=(\pi_{v}(g_{v})\psi_{v},\psi_{v})/(\psi_{v},\psi_{v})
\]
for $h_{v}\in\spp_{4n+4r}(\qq_{v}),\ g_{v}\in\spp_{2r}(\qq_{v})$
and $I(\Phi_{v},\Psi_{v})$ is defined by
\[
I(\Phi_{v},\Psi_{v})=\int_{\spp_{2r}(\qq_{v})}\Phi_{v}(g_{v})\overline{\Psi_{v}(g_{v})}dg_{v},
\]
where $dg_{v}$ is the Haar measure on $\spp_{2r}(\qq_{v})$ such
that 
\[
\mathrm{vol}(\spp_{2r}(\zz)\backslash\spp_{2r}(\rr),dg_{\infty})=\zeta(2)\zeta(4)\dots\zeta(2r),\ \mathrm{vol}(\spp_{2r}(\zz_{p}),dg_{p})=\zeta_{p}(2)^{-1}\zeta_{p}(4)^{-1}\dots\zeta_{p}(2r)^{-1}\ (p<\infty);
\]
note that $\prod_{v}dg_{v}$ is the Tamagawa measure on $\spp_{2r}(\aqq)$.
\end{itemize}
The aim of this section is to rewrite (\ref{eq:ikeconj}) into an
equivalent equation between the left hand side of (\ref{eq:conjggp})
and some L-values.

First, we rewrite (\ref{eq:ikeconj}) into an equation which is invariant
under scalar multiplication of $\mathbf{F}$ and $\mathbf{g}$. We
recall an important result of Katsurada and Kawamura.
\begin{thm}[{\cite[Theorem 2.3]{MR3384518}}]
\label{thm:-We-havewhere} We have
\begin{equation}
\frac{\langle\mathbf{f},\mathbf{f}\rangle\langle\mathbf{\mathbf{F}},\mathbf{\mathbf{F}}\rangle}{\langle\mathbf{h},\mathbf{h}\rangle}=2^{-a_{0}}\Lambda(k+n+r,f)\prod_{i=1}^{n+r}\Lambda(2i-1,f,\mathrm{Ad})\xi(2i)\Gamma_{\rr}(2i-1)\Gamma_{\rr}(2i+1),\label{eq:katsukawa}
\end{equation}
where $a_{0}=2k(n+r)+2(n+r)-k-1.$
\end{thm}

Dividing by (\ref{eq:katsukawa}), we see that (\ref{eq:ikeconj})
is equivalent to

\begin{align}
\frac{\langle\mathbf{\mathbf{F_{g}}},\mathbf{\mathbf{F_{g}}}\rangle}{\langle\mathbf{F},\mathbf{F}\rangle\langle\mathbf{g},\mathbf{g}\rangle} & =2^{-(b+c)}\frac{\Lambda(k+n,\mathrm{st}(g)\boxtimes f)}{\Lambda(k+n+r,f)\prod_{i=1}^{r}\Lambda(2n+2i-1,f,\mathrm{Ad})\xi(2n+2i)}\label{eq:ikeconj-1}\\
 & \times\left[\prod_{i=1}^{r}\Gamma_{\rr}(2n+2i-1)\Gamma_{\rr}(2n+2i+1)\right]^{-1},\nonumber 
\end{align}
where $b=r^{2}-r+2rn-1.$

Next, we adelize the left hand side of (\ref{eq:ikeconj-1}). For
$d\in\zz_{>0}$, put 
\[
\Delta_{\spp_{2d}}=\xi(2)\xi(4)\dots\xi(2d).
\]
Then, by \cite{MR1501843}, the volume of the fundamental domain of
$\mathfrak{h}_{d}$ with respect to $\spp_{d}(\zz)$ is $2\Delta_{\spp_{2d}}$.
Thus, since the Tamagawa number of $\spp_{2d}$ is $1$, we have
\[
\langle\mathbf{\mathbf{F}},\mathbf{\mathbf{F}}\rangle=2\Delta_{\spp_{4n+4r}}(\varphi,\varphi),\ \langle\mathbf{\mathbf{g}},\mathbf{\mathbf{g}}\rangle=2\Delta_{\spp_{2r}}(\psi,\psi),\ \langle\mathbf{\mathbf{F_{g}}},\mathbf{\mathbf{F_{g}}}\rangle=8\Delta_{\spp_{2r}}^{2}\Delta_{\spp_{4n+2r}}(\varphi_{\psi},\varphi_{\psi}).
\]
Therefore, (\ref{eq:ikeconj-1}) is equivalent to

\begin{align}
\frac{(\varphi_{\psi},\varphi_{\psi})}{(\varphi,\varphi)(\psi,\mathbf{\psi})} & =2^{-c}\frac{\Delta_{\spp_{4n+4r}}}{\Delta_{\spp_{2r}}\Delta_{\spp_{4n+2r}}}\label{eq:rere}\\
 & \times\frac{\Lambda(k+n,\mathrm{st}(g)\boxtimes f)}{\Lambda(k+n+r,f)\prod_{i=1}^{r}\Lambda(2n+2i-1,f,\mathrm{Ad})\xi(2n+2i)}\nonumber \\
 & \times2^{-(r^{2}-r+2rn)}\left[\prod_{i=1}^{r}\Gamma_{\rr}(2n+2i-1)\Gamma_{\rr}(2n+2i+1)\right]^{-1}.\nonumber 
\end{align}

\section{Statement of the main theorem\label{sec:Main-theorem}}

In this section, we state the main theorem.

Put 

\begin{align*}
L'_{\infty}(s,\mathrm{st}(g)\boxtimes f) & =\Gamma_{\cc}(s)\prod_{i=1}^{r}\Gamma_{\cc}(s+n-k+i)\Gamma_{\cc}(s+n+k+i-1),\\
L'_{\infty}(s,f,\mathrm{Ad}) & =\Gamma_{\rr}(s+1)\Gamma_{\cc}(s+2k-1).
\end{align*}
Then, $\Lambda(s,\mathrm{st}(g)\boxtimes f)$ and $\Lambda(s,f,\mathrm{Ad})$
are defined by 

\begin{align*}
\Lambda(s,\mathrm{st}(g)\boxtimes f) & =L'_{\infty}(s,\mathrm{st}(g)\boxtimes f)\prod_{p<\infty}L(s-k+1/2,\pi_{p}\times\sigma_{p}),\\
\Lambda(s,f,\mathrm{Ad}) & =L'_{\infty}(s,f,\mathrm{Ad})\prod_{p<\infty}L(s,\sigma_{p},\mathrm{Ad})
\end{align*}
(see \cite{MR2219248}). Put 
\[
\mathcal{L}=\frac{\Delta_{\spp_{4n+4r}}}{\Delta_{\spp_{2r}}\Delta_{\spp_{4n+2r}}}\frac{L(n+1/2,\pi\times\sigma)}{L(n+r+1/2,\sigma)\prod_{i=1}^{r}L(2n+2i-1,\sigma,\mathrm{Ad})\xi(2n+2i)}
\]
 and 

\[
\mathcal{L}_{v}=\frac{(\Delta_{\spp_{4n+4r}})_{v}}{(\Delta_{\spp_{2r}})_{v}(\Delta_{\spp_{4n+2r}})_{v}}\frac{L(n+1/2,\pi_{v}\times\sigma_{v})}{L(n+r+1/2,\sigma_{v})\prod_{i=1}^{r}L(2n+2i-1,\sigma_{v},\mathrm{Ad})\zeta_{v}(2n+2i)}
\]
for any place $v$ of $\qq$, where $\zeta_{\infty}(s)=\Gamma_{\rr}(s)$
and $(\Delta_{\spp_{2d}})_{v}=\zeta_{v}(2)\zeta_{v}(4)\dots\zeta_{v}(2d)$.
For any finite set $S$ of places of $\qq$ containing $\infty$,
put

\[
\mathcal{L}^{S}=\mathcal{L}\prod_{v\in S}\mathcal{L}_{v}^{-1}.
\]
Moreover, put

\begin{align*}
\mathcal{L}'_{\infty} & =\frac{(\Delta_{\spp_{4n+4r}})_{\infty}}{(\Delta_{\spp_{2r}})_{\infty}(\Delta_{\spp_{4n+2r}})_{\infty}}\frac{L'_{\infty}(k+n,\mathrm{st}(g)\boxtimes f)}{L'_{\infty}(k+n+r,f)\prod_{i=1}^{r}L'_{\infty}(2n+2i-1,f,\mathrm{Ad})\zeta_{\infty}(2n+2i)}.
\end{align*}

\begin{rem}
We did not fix the definitions of $L(s,\pi_{v}\times\tau_{v})$ and
$L(s,\sigma_{v},\mathrm{Ad})$ explicitly. Therefore, although $\mathcal{L}_{p}$
is uniquely determined since $\pi_{p}$ and $\sigma_{p}$ are unramified,
$\mathcal{L}_{\infty}$ could have an ambiguity depending on them.
To avoid this ambiguity, $\mathcal{L}'_{\infty}$ was introduced.
\end{rem}

Then, the main theorem of this paper is the following.
\begin{thm}
\label{thm:Assume-that-}Let $v$ be a place of $\qq$. Assume that
$\pi_{p}$ is tempered if $v=p<\infty$. Then, we have

\[
I(\Phi_{v},\Psi_{v})=\begin{cases}
\mathcal{L}'_{\infty}\times2^{-(r^{2}-r+2rn)}\left[\prod_{i=1}^{r}\Gamma_{\rr}(2n+2i-1)\Gamma_{\rr}(2n+2i+1)\right]^{-1} & v=\infty,\\
\mathcal{L}_{p} & v=p<\infty.
\end{cases}
\]
\end{thm}

Combining Theorem \ref{thm:Assume-that-} with (\ref{eq:rere}), we
obtain the equivalence between Conjecture \ref{conj:Assume-that-.}
and Conjecture \ref{conj:Assume-that-}. We will prove this theorem
in the next section.

\section{Proof of Theorem \ref{thm:Assume-that-}\label{sec:Proof-of-theorem}}

We prove Theorem \ref{thm:Assume-that-} here. In this section, fix
a place $v$ of $\qq$.

\subsection{real case\label{subsec:real-case}}

Here we assume that $v=\infty$. 

First, we recall the classical measure $d_{C}g$ on $\spp_{2d}(\rr)$
(cf. \cite{MR4229179}). $d_{C}g$ is the Haar measure on $\spp_{2d}(\rr)$
satisfying

\[
\int_{\spp_{2d}(\rr)}f(g(i1_{d}))d_{C}g=\int_{\mathfrak{h}_{d}}f(X+iY)\det(Y)^{-(d+1)}dXdY
\]
for any function $f$ on $\mathfrak{h}_{d}$. For a fundamental domain
$\mathfrak{F}$ of $\mathfrak{h}_{d}$ with respect to $\spp_{2d}(\zz)$,
\[
\int_{\mathfrak{F}}\det(Y)^{-(d+1)}dXdY=2\Delta_{\spp_{2d}}
\]
holds (see \cite{MR1501843}).
\begin{lem}
\label{lem:vol}We have $\mathrm{vol}(\spp_{2d}(\zz)\backslash\spp_{2d}(\rr),d_{C}g)=\Delta_{\spp_{2d}}$.
\end{lem}

\begin{proof}
Put $\mathfrak{S}=\{g\in\spp_{2d}(\rr)\ |\ g(i1_{d})\in\mathfrak{F}\}$.
Then, 
\[
\int_{\mathfrak{F}}\det(Y)^{-(d+1)}dXdY=\int_{\spp_{2d}(\rr)}1_{\mathfrak{S}}(g)d_{C}g=\int_{\spp_{2d}(\zz)\backslash\spp_{2d}(\rr)}\sum_{\gamma\in\spp_{2d}(\zz)}1_{\mathfrak{S}}(\gamma g)d_{C}g.
\]
For any $g\in\spp_{2d}(\rr)$, $\#\{\gamma\in\spp_{2d}(\zz)\ |\ \gamma g\in\mathfrak{S}\}=2$
holds. Thus we have 
\[
\int_{\spp_{2d}(\zz)\backslash\spp_{2d}(\rr)}d_{C}g=2^{-1}\int_{\mathfrak{F}}\det(Y)^{-(d+1)}dXdY=\Delta_{\spp_{2d}}.
\]
\end{proof}
Next, we prepare a fact of matrix coefficients of lowest weight modules
of scalar type. For $w,\ d\in\zz_{>0},$ denote by $\tau_{w,d}$ the
lowest weight module of $\spp_{2d}(\rr)$ with minimal $\bu(d)$-type
$\det^{w}$.
\begin{lem}[{\cite[Proposition A.1]{MR3960116}, \cite[(148)]{MR4229179}}]
\label{lem:Let--be}Assume that $w>d$ and let $v_{0}\in\tau=\tau_{w,d}$
be a lowest weight vector such that $(v_{0},v_{0})_{\tau}=1.$ Then,
we have 
\[
(\tau(g)v_{0},v_{0})_{\tau}=\frac{2^{dw}}{\det(A+D+i(-B+C))^{w}},
\]
where $g=\begin{pmatrix}A & B\\
C & D
\end{pmatrix}\in\spp_{2d}(\rr)$. Moreover,
\begin{align*}
\int_{\spp_{2d}(\rr)}|(\tau(g)v_{0},v_{0})_{\tau}|^{2}d_{C}g & (=2^{d(d+3)/2}\pi^{d(d+1)/2}\prod_{1\leq i\leq d}\frac{\Gamma(w-(d+i)/2)}{\Gamma(w-(d-i)/2)}\\
 & =2^{d(d+3)/2}\prod_{1\leq i\leq d}\frac{\Gamma_{\rr}(2w-(d+i))}{\Gamma_{\rr}(2w-(d-i))})\\
 & =2^{d(d+3)/2}\prod_{1\leq i\leq d}\frac{\Gamma_{\rr}(2w-2d+i-1)}{\Gamma_{\rr}(2w-d+i)}
\end{align*}
holds.
\end{lem}

Then we can prove the first half of Theorem \ref{thm:Assume-that-}
as follows.
\begin{proof}[Proof of Theorem \ref{thm:Assume-that-} for $v=\infty$]
Let $v_{0}\in\tau_{k+n+r,2n+2r}$ be a nonzero lowest weight vector.
Then, the $\spp_{2r}(\rr)$-subrepresentation of $\tau_{k+n+r,2n+2r}$
generated by $v_{0}$ is equal to $\tau_{k+n+r,r}$ with lowest weight
vector $v_{0}$. Therefore, we have $\Phi_{\infty}|_{\spp_{2r}(\rr)}=\Psi_{\infty}.$

By Lemma \ref{lem:vol} and Lemma \ref{lem:Let--be}, we have

\begin{align}
I(\Phi_{\infty},\Psi_{\infty}) & =\int_{\spp_{2r}(\rr)}|\Psi_{\infty}(g)|^{2}dg_{\infty}\label{eq:fd}\\
 & =(\Delta_{\spp_{2r}})_{\infty}^{-1}2^{r(r+3)/2}\prod_{i=1}^{r}\frac{\Gamma_{\rr}(2k+2n+i-1)}{\Gamma_{\rr}(2k+2n+r+i)}.\nonumber 
\end{align}
On the other hand, 
\begin{align*}
 & L'_{\infty}(k+n+r,f)\prod_{i=1}^{r}L'_{\infty}(2n+2i-1,f,\mathrm{Ad})\zeta_{\infty}(2n+2i)\Gamma_{\rr}(2n+2i-1)\Gamma_{\rr}(2n+2i+1)\\
= & \Gamma_{\cc}(n+r+k)\prod_{i=1}^{r}\Gamma_{\cc}(2n+2k+2i-2)\Gamma_{\rr}(2n+2i-1)\Gamma_{\rr}(2n+2i)^{2}\Gamma_{\rr}(2n+2i+1)\\
= & \Gamma_{\cc}(n+r+k)\prod_{i=1}^{r}\Gamma_{\cc}(2n+2k+2i-2)\Gamma_{\cc}(2n+2i-1)\Gamma_{\cc}(2n+2i)\\
= & \Gamma_{\cc}(n+r+k)\prod_{i=1}^{r}\Gamma_{\cc}(2n+2k+2i-2)\Gamma_{\cc}(2n+i)\Gamma_{\cc}(2n+r+i).
\end{align*}
Therefore, we have 
\begin{align*}
 & \mathcal{L}'_{\infty}\times\left[\prod_{i=1}^{r}\Gamma_{\rr}(2n+2i-1)\Gamma_{\rr}(2n+2i+1)\right]^{-1}\\
= & (\Delta_{\spp_{2r}})_{\infty}^{-1}\frac{\Gamma_{\cc}(n+k)}{\Gamma_{\cc}(n+r+k)}\prod_{i=1}^{r}\frac{\Gamma_{\cc}(2n+i)\Gamma_{\cc}(2n+2k+i-1)\Gamma_{\rr}(4n+2r+2i)}{\Gamma_{\cc}(2n+2k+2i-2)\Gamma_{\cc}(2n+i)\Gamma_{\cc}(2n+r+i)}\\
= & (\Delta_{\spp_{2r}})_{\infty}^{-1}\frac{\Gamma_{\cc}(n+k)}{\Gamma_{\cc}(n+r+k)}\prod_{i=1}^{r}\frac{\Gamma_{\cc}(2n+2k+i-1)}{\Gamma_{\cc}(2n+2k+2i-2)}\frac{\Gamma_{\rr}(4n+2r+2i)}{\Gamma_{\cc}(2n+r+i)}\\
= & (\Delta_{\spp_{2r}})_{\infty}^{-1}\frac{\Gamma_{\cc}(n+k)}{\Gamma_{\cc}(n+r+k)}\prod_{i=1}^{r}\frac{\Gamma_{\rr}(2n+2k+i-1)}{\Gamma_{\rr}(2n+2k+r+i)}\frac{\Gamma_{\rr}(2n+2k+r+i)\Gamma_{\rr}(2n+2k+i)}{\Gamma_{\rr}(2n+2k+2i-2)\Gamma_{\rr}(2n+2k+2i-1)}2^{2n+r+i-1}
\end{align*}
by $\Gamma_{\rr}(2s)/\Gamma_{\cc}(s)=2^{s-1}.$ Then we have

\begin{align*}
 & \mathcal{L}'_{\infty}\times\left[\prod_{i=1}^{r}\Gamma_{\rr}(2n+2i-1)\Gamma_{\rr}(2n+2i+1)\right]^{-1}\\
= & I(\Phi_{\infty},\Psi_{\infty})2^{-r(r+3)/2}2^{2nr+r^{2}+r(r+1)/2-r}\frac{\Gamma_{\cc}(n+k)}{\Gamma_{\cc}(n+r+k)}\prod_{i=1}^{r}\frac{\Gamma_{\rr}(2n+2k+r+i)\Gamma_{\rr}(2n+2k+i)}{\Gamma_{\rr}(2n+2k+2i-2)\Gamma_{\rr}(2n+2k+2i-1)}\\
= & I(\Phi_{\infty},\Psi_{\infty})2^{2nr+r^{2}-2r}\frac{\Gamma_{\cc}(n+k)}{\Gamma_{\cc}(n+r+k)}\prod_{i=1}^{r}\frac{\Gamma_{\rr}(2n+2k+2i)\Gamma_{\rr}(2n+2k+2i-1)}{\Gamma_{\rr}(2n+2k+2i-2)\Gamma_{\rr}(2n+2k+2i-1)}\\
= & I(\Phi_{\infty},\Psi_{\infty})2^{2nr+r^{2}-2r}\frac{\Gamma_{\cc}(n+k)}{\Gamma_{\cc}(n+r+k)}\frac{\Gamma_{\rr}(2n+2k+2r)}{\Gamma_{\rr}(2n+2k)}\\
= & I(\Phi_{\infty},\Psi_{\infty})2^{2nr+r^{2}-2r}2^{1-(n+k)-1+(n+r+k)}=I(\Phi_{\infty},\Psi_{\infty})2^{2nr+r^{2}-r}
\end{align*}
by (\ref{eq:fd}) . Thus we obtain 
\[
I(\Phi_{\infty},\Psi_{\infty})=\mathcal{L}'_{\infty}\times2^{-(r^{2}-r+2rn)}\left[\prod_{i=1}^{r}\Gamma_{\rr}(2n+2i-1)\Gamma_{\rr}(2n+2i+1)\right]^{-1}.
\]
\end{proof}

\subsection{$p$-adic case}

Here we assume that $v=p<\infty$. We often drop $p$ subscript for
simplicity.

Before starting, we introduce some additional notations:
\begin{itemize}
\item From now on, we redefine $\spp_{2d}(R)$ by
\[
\spp_{2d}(R)=\left\{ g\in\gll_{2d}(R)\ \middle|\ g\begin{pmatrix} & -J_{d}\\
J_{d}
\end{pmatrix}{}^{t}g=\begin{pmatrix} & -J_{d}\\
J_{d}
\end{pmatrix}\right\} 
\]
for any ring $R$, where $J_{d}=\begin{pmatrix} &  & 1\\
 & \stodd\\
1
\end{pmatrix}\in\gll_{d}(\zz)$. Then, $\spp_{2a}\times\spp_{2b}$ is identified with the subgroup
\[
\left\{ \begin{pmatrix}A &  & B\\
 & g\\
C &  & D
\end{pmatrix}\ \middle|\ \begin{pmatrix}A & B\\
C & D
\end{pmatrix}\in\spp_{2a},g\in\spp_{2b}\right\} 
\]
of $\spp_{2d}$, where $a+b=d$. 
\item We put $F=\qq_{p},$ $\mathcal{O}=\zz_{p}$ and denote by $\abs$
the normalized absolute value of $F$.
\item For $X\in\mathrm{M}_{d,e}(F)$, put $X^{*}=J_{e}{}^{t}XJ_{d}$.
\item For $g\in\mathrm{\gll}_{d}(F)$, put $m(g)=\begin{pmatrix}g\\
 & (g^{*})^{-1}
\end{pmatrix}\in\spp_{2d}(F)$.
\item For any representation $\tau$ of $\gll_{d}(F)$ and character $\lambda$
of $F^{\times},$ put $\tau\lambda=\tau\otimes\lambda(\det)$ for
short.
\item Let $\tau_{1},\dots,\tau_{N}$ be representations of $\gll_{d_{1}}(F),\dots,\gll_{d_{N}}(F)$,
respectively. Then we denote the normalized parabolically induced
representation
\[
\ind_{P(F)}^{\gll_{d'}(F)}\tau_{1}\boxtimes\dots\boxtimes\tau_{N}
\]
by 
\[
\tau_{1}\times\dots\times\tau_{N},
\]
 where $d'=d_{1}+\dots+d_{N}$ and $P$ is the block upper triangular
parabolic subgroup of $\gll_{d'}$ whose Levi subgroup is $\diag(\gll_{d_{1}},\dots,\gll_{d_{N}}).$
Moreover, for a representation $\mu$ of $\spp_{2d_{0}}(F)$ such
that $d_{0}+d'=d,$ we denote the normalized parabolically induced
representation
\[
\ind_{Q(F)}^{\spp_{2d}(F)}\tau_{1}\boxtimes\dots\boxtimes\tau_{N}\boxtimes\mu
\]
by 
\[
\tau_{1}\times\dots\times\tau_{N}\rtimes\mu,
\]
where $Q$ is the block upper triangular parabolic subgroup of $\spp_{2d}$
whose Levi subgroup is naturally isomorphic to $\gll_{d_{1}}\times\dots\times\gll_{d_{N}}\times\spp_{2d_{0}}$.
\end{itemize}
First, we check that $I(\Phi,\Psi)$ is well-defined. By the asymptotic
behavior of matrix coefficients (see \cite[\S 4]{padicbook}) and
the geometric lemma (see \cite[\S 2]{zbMATH03639915}), we obtain
the following. 
\begin{lem}
\label{lem:Let--and}Let $\lambda_{1},\dots,\lambda_{d}$ be unitary
characters of $F^{\times}$ and $r_{1},\dots,r_{d}\in\rr$. Put 
\[
\tau=\lambda_{1}\abs^{r_{1}}\times\dots\times\lambda_{d}\abs^{r_{d}}\rtimes1_{\spp_{0}}.
\]
 Let $f$ be a matrix coefficient of $\tau$. Then, for any $\ep>0$,
there exist $C_{\ep}>0$ such that 
\[
|f(m(\diag(a_{1},\dots,a_{d}))|\leq C_{\ep}\delta_{d}(a_{1},\dots,a_{d})^{1/2}|a_{1}\dots a_{d}|^{-R-\ep}
\]
for any $a_{1},\dots,a_{d}\in F^{\times}$ such that $|a_{1}|\leq\dots\leq|a_{d}|\leq1,$
where 
\[
\delta_{d}(a_{1},\dots,a_{d})=|a_{1}|^{2d}|a_{2}|^{2d-2}\dots|a_{d}|^{2},\ R=\max\{|r_{i}|\ |\ i=1,\dots,d\}.
\]
\end{lem}

For convenience, we denote $\sigma=\chi\times\chi^{-1},$ where $\chi$
is an unramified character of $F^{\times}$. We note that $\chi$
is unitary by Ramanujan conjecture. It is known that
\[
\Sigma=\chi(\det{}_{2n+2r})\rtimes1_{\spp_{0}}
\]
(see \cite{MR1884618}, note that $\chi(\det{}_{2n+2r})\rtimes1_{\spp_{0}}$
is irreducible by \cite{MR1194967}).
\begin{cor}
\label{cor:If--is}If $\pi$ is tempered, then $I(\Phi,\Psi)$ is
well-defined.
\end{cor}

\begin{proof}
By the Cartan decomposition, we have 
\begin{align*}
 & \int_{\spp_{2r}(F)}|\Phi(g)\overline{\Psi(g)}|dg\\
\leq & C\int_{|a_{1}|\leq\dots\leq|a_{r}|\leq1}\delta_{r}(a_{1},\dots,a_{r})^{-1}|\Phi(m(\diag(a_{1},\dots,a_{r},1_{2n+r})))\Psi(m(\diag(a_{1},\dots,a_{r})))|d^{\times}a_{1}\dots d^{\times}a_{r}
\end{align*}

for some $C>0$. By Lemma \ref{lem:Let--and}, for any $\ep>0$, there
exist $C_{\ep}>0$ such that

\begin{align*}
|\Phi(m(\diag(a_{1},\dots,a_{r},1_{2n+r})))| & \leq C_{\ep}|a_{1}|^{2n+2r}\dots|a_{r}|^{2n+r+1}|a_{1}\dots a_{r}|^{-\ep-(2n+2r-1)/2}\\
 & =C_{\ep}\delta_{r}(a_{1},\dots,a_{r})^{1/2}|a_{1}\dots a_{r}|^{n+1/2-\ep}
\end{align*}
and

\[
|\Psi(m(\diag(a_{1},\dots,a_{r})))|\leq C_{\ep}\delta_{r}(a_{1},\dots,a_{r})^{1/2}|a_{1}\dots a_{r}|^{-\ep}
\]
if $|a_{1}|\leq\dots\leq|a_{r}|\leq1$, since $\pi$ is tempered and
$\Sigma\subset\chi\abs^{-(2n+2r-1)/2}\times\dots\times\chi\abs^{(2n+2r-1)/2}\rtimes1_{\spp_{0}}.$
Thus we have
\[
\int_{\spp_{2r}(F)}|\Phi(g)\overline{\Psi(g)}|dg\leq CC_{\ep}^{2}\prod_{i=1}^{r}\int_{\mathcal{O}}|a_{i}|^{n+1/2-2\ep}d^{\times}a_{i}.
\]
Since $n+1/2-2\ep>0$ if we take $\ep<1/4,$ $I(\Phi,\Psi)$ is well-defined.
\end{proof}
Next, we recall local twisted doubling zeta integrals (see \cite{Cai:2019aa}).
For any $d\in\zz_{>0}$, we denote the unique irreducible subrepresentation
of $\sigma\abs^{-(d-1)/2}\times\dots\times\sigma\abs^{(d-1)/2}$ by
$\sigma_{d}$. It is known that there exists a unique realization
of $\sigma_{d}$ in $\ind_{U_{\gll}}^{\gll_{2d}(F)}\psi_{F}\circ\tr,$
where $\psi_{F}$ is a nontrivial additive character of $F$ and $U_{\gll}=\left\{ \begin{pmatrix}1_{d} & *\\
 & 1_{d}
\end{pmatrix}\right\} \simeq\mathrm{M}_{d}(F)$. We denote it by $\mathcal{W}(\sigma_{d})$. Note that the following
equation
\[
W(\begin{pmatrix}a\\
 & a
\end{pmatrix}g)=W(g)
\]
holds for any $W\in\mathcal{W}(\sigma_{d}),\ g\in\gll_{2d}(F),$ and
$a\in\gll_{d}(F).$ For simplicity, we fix $\psi_{F}$ so that $\psi_{F}(x)=e^{2\pi ix}$
in this paper.

Put 
\[
U=\left\{ \begin{pmatrix}1_{2m} & * & *\\
 & 1_{4m} & *\\
 &  & 1_{2m}
\end{pmatrix}\right\} \subset\spp_{8m}(F)
\]
 and $\psi_{U}(u)=\psi_{F}(\tr(X+Y))$ for
\[
u=\begin{pmatrix}1_{m} &  & X & * & * & * & * & *\\
 & 1_{m} & * & * & * & Y & * & *\\
 &  & 1_{m} &  &  &  & * & *\\
 &  &  & 1_{m} &  &  & * & *\\
 &  &  &  & 1_{m} &  & * & *\\
 &  &  &  &  & 1_{m} & * & *\\
 &  &  &  &  &  & 1_{m}\\
 &  &  &  &  &  &  & 1_{m}
\end{pmatrix}\in U.
\]
For a standard section $f_{s}$ of $\mathcal{W}(\sigma_{2m})\abs^{s}\rtimes1_{\spp_{0}}$,
define the Fourier coefficient $\mathcal{F}(f_{s})$ of $f_{s}$ with
respect to $\psi_{U}$, which is a function on $\spp_{2m}(F)\times\spp_{2m}(F)$,
by

\[
\mathcal{F}(f_{s})(g,h)=\int_{U_{0}}f_{s}(\delta u_{0}i(g,h))\psi_{U}(u_{0})du_{0}.
\]
Moreover, for a matrix coefficient $\omega$ of the dual $\mu^{\vee}$
of an irreducible representation $\mu$ of $\spp_{2m}(F)$, define
the local twisted doubling zeta integral $Z(f_{s},\omega)$ by

\[
Z(f_{s},\omega)=\int_{\spp_{2m}(F)}\mathcal{F}(f_{s})(1,g^{\iota})\omega(g)dg,
\]
where 

\[
U_{0}=U\cap\left\{ \begin{pmatrix}1_{8m} & *\\
 & 1_{8m}
\end{pmatrix}\right\} ,\ \delta=\begin{pmatrix} & 1_{4m}\\
-1_{4m}
\end{pmatrix}\begin{pmatrix}1_{2m}\\
 & 1_{2m} & 1_{2m}\\
 &  & 1_{2m}\\
 &  &  & 1_{2m}
\end{pmatrix},
\]
 $i$ is the embedding of $\spp_{2m}(F)\times\spp_{2m}(F)$ in $\spp_{8m}(F)$
defined by 
\[
i(g,h)=\begin{pmatrix}g\\
 & A &  & B\\
 &  & h\\
 & C &  & D\\
 &  &  &  & g^{*-1}
\end{pmatrix}\mbox{ for }g=\begin{pmatrix}A & B\\
C & D
\end{pmatrix},\ h\in\spp_{2m}(F),
\]
 and $du_{0}$ is normalized by $\mathrm{vol}(U_{0}\cap\gll_{8m}(\mathcal{O}),du_{0})=1$.
\begin{rem}
\label{rem:We-note-some}We note some properties of $\mathcal{F}(f_{s})$
and $Z(f_{s},\omega)$:
\begin{itemize}
\item If $\mathrm{Re}(s)$ is sufficiently large, for any $f_{s}$ and $\omega$,
the integral defining $\mathcal{F}(f_{s})$ and $Z(f_{s},\omega)$
are absolutely convergent. Moreover, all $\mathcal{F}(f_{s})$ and
$Z(f_{s},\omega)$ admit meromorphic continuations and there are polynomials
$P(X),\ Q(X)\in\cc[X]$ such that all $P(q^{-s})\mathcal{F}(f_{s})(1,1)$
and all $Q(q^{-s})Z(f_{s},\omega)$ are in $\cc[q^{\pm s}]$.
\item $i(\spp_{2m}(F),\spp_{2m}(F))$ stabilize $\psi_{U}$ and $\mathcal{F}(f_{s}(\cdot\,u))(g,h)=\psi_{U}^{-1}(u)\mathcal{F}(f_{s})(g,h)$
for any $u\in U.$
\item $\mathcal{F}(f_{s})(g_{0}g,g_{0}^{\iota}h)=\mathcal{F}(f_{s})(g,h)$
for any $g_{0}\in\spp_{2m}(F)$, where $g_{0}^{\iota}=\begin{pmatrix} & 1_{m}\\
1_{m}
\end{pmatrix}g_{0}\begin{pmatrix} & 1_{m}\\
1_{m}
\end{pmatrix};$ denote $\Delta(g_{0})=i(g_{0},g_{0}^{\iota})$ for short.
\item If all data are unramified, $\vol(\spp_{2m}(\mathcal{O}),dg)=\Delta_{\spp_{2m}}^{-1}$,
and $f_{s}(1)=\omega(1)=1,$ then 
\[
Z(f_{s},\omega)=\Delta_{\spp_{2m}}^{-1}\frac{L(s+1/2,\mu\times\sigma)}{L(s+m+1/2,\sigma)\prod_{i=1}^{m}L(2s+2i-1,\sigma,{\rm Ad})\zeta(2s+2i)}.
\]
(Note: $L(s,\sigma,{\rm Sym}^{2})=L(s,\sigma,{\rm Ad})=\zeta(s)L(s,\chi^{2})L(s,\chi^{-2}),\ L(s,\sigma,{\rm \wedge}^{2})=\zeta(s)$.)
\end{itemize}
\end{rem}

For any $d\in\zz_{>0}$, let $f_{s}^{d}$ be the section of $\mathcal{W}(\sigma_{d})\abs^{s}\rtimes1_{\spp_{0}}$
such that $f_{s}^{d}|_{\spp_{4d}(\mathcal{O})}\equiv1$. Define $F_{s,m}$
by $F_{s,m}(g)=\mathcal{F}(f_{s}^{2m})(1,g^{\iota})$ for $g\in\spp_{2m}(F).$
$F_{s,m}$ has the following inductivity.
\begin{prop}
\label{prop:Assume-that-} Let $m_{0}\in\zz_{>0}$ such that $m-m_{0}\geq0$.
Then, for any $g\in\spp_{2(m-m_{0})}(F)$, we have

\begin{equation}
F_{s,m}(\diag(1_{m_{0}},g,1_{m_{0}}))=F_{s+m_{0},m-m_{0}}(g)\zeta(4s+2m+2)^{-1}\zeta(4s+2m+4)^{-1}\dots\zeta(4s+2m+2m_{0})^{-1}.\label{eq:ind}
\end{equation}
In particular, $F_{-m/2,m}(1)=\Delta_{\spp_{2m}}^{-1}$.
\end{prop}

\begin{proof}
In this proof, the Haar measure of any topological group $H$ such
that $H\subset\gll_{d}(F)$ for some $d\in\zz_{>0}$ is normalized
so that the volume of $H\cap\gll_{d}(\mathcal{O})$ is $1$. In addition,
for any $d\in\zz_{>0},$ the Haar measure of $F^{d}$ is normalized
so that the volume of $\mathcal{O}^{d}$ is $1$.

It is sufficient to show (\ref{eq:ind}) for $m_{0}=1$ by induction,
and we can assume that $g=m(A)$ by the Cartan decomposition, where
$A\in\gll_{m-1}(F)\cap\mathrm{M}_{m-1}(\mathcal{O})$.

Put $I=F_{s,m}(\diag(1,g,1))$ for short. Since $f_{s}^{2m}$ is $\spp_{8m}(\mathcal{O})$-invariant,
we have

\begin{align*}
I & =\int_{U_{0}}f_{s}^{2m}(\delta u_{0}i(1,m(\diag(A^{*-1},1))))\psi_{U}(u_{0})du_{0}\\
 & =\int_{U_{0}^{1}}f_{s}^{2m}(u_{0}^{1}m(\diag(1,A,1_{3m})))\psi_{U}^{1}(u_{0}^{1})du_{0}^{1},
\end{align*}
where 
\[
U_{0}^{1}=\left\{ \begin{pmatrix}1_{2m}\\
 & 1_{2m}\\
* & * & 1_{2m}\\
 & * &  & 1_{2m}
\end{pmatrix}\right\} 
\]
 and $\psi_{U}^{1}(u_{0}^{1})=\psi_{F}^{-1}(\tr X\begin{pmatrix}0_{m}\\
 & 1_{m}
\end{pmatrix})$ for 
\[
u_{0}^{1}=\begin{pmatrix}1_{2m}\\
 & 1_{2m}\\
X & * & 1_{2m}\\
 & * &  & 1_{2m}
\end{pmatrix}.
\]
 Moreover, for $k=\begin{pmatrix}1_{m}\\
 &  & 1_{m-1}\\
 & 1
\end{pmatrix},$ we have
\begin{align*}
I & =\int_{U_{0}^{1}}f_{s}^{2m}(u_{0}^{1}m(\diag(1,A,1_{3m})\diag(k,k)))\psi_{U}^{1}(u_{0}^{1})du_{0}^{1}\\
 & =\int_{U_{0}^{1}}f_{s}^{2m}(u_{0}^{1}m(\diag(1,A,1_{3m})))\psi_{U}^{2}(u_{0}^{1})du_{0}^{1},
\end{align*}
where $\psi_{U}^{2}(u_{0}^{1})=\psi_{F}^{-1}(\tr\begin{pmatrix}0_{m}\\
 & 1_{m-1}
\end{pmatrix}Y+\tr W\begin{pmatrix}0^{m-1}\\
1\\
0^{m-1}
\end{pmatrix})$ for 
\[
u_{0}^{1}=\begin{pmatrix}1\\
 & 1_{2m-1}\\
 &  & 1\\
 &  &  & 1_{2m-1}\\
X & Y & P & Q & 1_{2m-1}\\
Z & W & R & P^{*} &  & 1\\
 &  & W^{*} & Y^{*} &  &  & 1_{2m-1}\\
 &  & Z & X^{*} &  &  &  & 1
\end{pmatrix}.
\]
 We consider

\begin{align*}
F_{1}(Z) & =\int_{U_{0}^{2}}f_{s}^{2m}(\begin{pmatrix}1\\
 & 1_{2m-1}\\
 &  & 1\\
 &  &  & 1_{2m-1}\\
 &  &  &  & 1_{2m-1}\\
Z &  &  &  &  & 1\\
 &  &  &  &  &  & 1_{2m-1}\\
 &  & Z &  &  &  &  & 1
\end{pmatrix}u_{0}^{2}m(\diag(1,A,1_{3m})))\psi_{U}^{2}(u_{0}^{2})du_{0}^{2}\\
 & =\int_{U_{0}^{2}}f_{s}^{2m}(u_{0}^{2}m(\diag(1,A,1_{3m}))\begin{pmatrix}1\\
 & 1_{2m-1}\\
 &  & 1\\
 &  &  & 1_{2m-1}\\
 &  &  &  & 1_{2m-1}\\
Z &  &  &  &  & 1\\
 &  &  &  &  &  & 1_{2m-1}\\
 &  & Z &  &  &  &  & 1
\end{pmatrix})\psi_{U}^{2}(u_{0}^{2})du_{0}^{2},
\end{align*}
where 
\[
U_{0}^{2}=\left\{ \begin{pmatrix}1\\
 & 1_{2m-1}\\
 &  & 1\\
 &  &  & 1_{2m-1}\\
* & * & * & * & 1_{2m-1}\\
 & * & * & * &  & 1\\
 &  & * & * &  &  & 1_{2m-1}\\
 &  &  & * &  &  &  & 1
\end{pmatrix}\right\} .
\]
If $|Z|\leq1$, we have $F_{1}(Z)=F_{1}(0)$ since $f_{s}^{2m}$ is
unramified$.$ On the other hand, if $|Z|>1$, for $a\in\mathcal{O}$
such that $\psi_{F}(-Za)\neq1$, we have

\begin{align*}
F_{1}(Z) & =\int_{U_{0}^{2}}f_{s}^{2m}(\begin{pmatrix}1\\
 & 1_{2m-1}\\
 &  & 1\\
 &  &  & 1_{2m-1}\\
 &  &  &  & 1_{2m-1}\\
Z &  &  &  &  & 1\\
 &  &  &  &  &  & 1_{2m-1}\\
 &  & Z &  &  &  &  & 1
\end{pmatrix}u_{0}^{2}m(\diag(1,A,1_{3m}))\begin{pmatrix}1 &  & a\\
 & 1_{8m-2}\\
 &  & 1
\end{pmatrix})\psi_{U}^{2}(u_{0}^{2})du_{0}^{2}\\
 & =\psi_{F}(-Za)F_{1}(Z)
\end{align*}
and $F_{1}(Z)=0$ since 
\begin{align*}
 & \begin{pmatrix}1 &  & -a\\
 & 1_{8m-2}\\
 &  & 1
\end{pmatrix}\begin{pmatrix}1\\
 & 1_{2m-1}\\
 &  & 1\\
 &  &  & 1_{2m-1}\\
X & Y & P & Q & 1_{2m-1}\\
Z & W & R & P^{*} &  & 1\\
 &  & W^{*} & Y^{*} &  &  & 1_{2m-1}\\
 &  & Z & X^{*} &  &  &  & 1
\end{pmatrix}\begin{pmatrix}1 &  & a\\
 & 1_{8m-2}\\
 &  & 1
\end{pmatrix}\\
 & =\begin{pmatrix}1 &  & -aZ & -aX\\
 & 1_{2m-1}\\
 &  & 1\\
 &  &  & 1_{2m-1}\\
X & Y & P & Q & 1_{2m-1} &  &  & aX\\
Z & W & R & P^{*} &  & 1 &  & aZ\\
 &  & W^{*} & Y^{*} &  &  & 1_{2m-1}\\
 &  & Z & X^{*} &  &  &  & 1
\end{pmatrix}\\
 & =\begin{pmatrix}1 &  & -aZ & -aX^{*}\\
 & 1_{2m-1}\\
 &  & 1\\
 &  &  & 1_{2m-1}\\
 &  &  &  & 1_{2m-1} &  &  & aX\\
 &  &  &  &  & 1 &  & aZ\\
 &  &  &  &  &  & 1_{2m-1}\\
 &  &  &  &  &  &  & 1
\end{pmatrix}\begin{pmatrix}1\\
 & 1_{2m-1}\\
 &  & 1\\
 &  &  & 1_{2m-1}\\
X & Y & P+* & Q+* & 1_{2m-1}\\
Z & W & R & P^{*}+* &  & 1\\
 &  & W^{*} & Y^{*} &  &  & 1_{2m-1}\\
 &  & Z & X^{*} &  &  &  & 1
\end{pmatrix}.
\end{align*}
 Therefore, we have
\[
I=\int_{U_{0}^{2}}f_{s}^{2m}(u_{0}^{2}m(\diag(1,A,1_{3m})))\psi_{U}^{2}(u_{0}^{2})du_{0}^{2}.
\]

Let $f_{s}^{1,2m-1}$ be the section of $(\mathcal{W}(\sigma)\abs^{-(2m-1)/2}\times\mathcal{W}(\sigma_{2m-1})\abs^{1/2})\abs^{s}\rtimes1_{\spp_{0}}$
satisfying $f_{s}^{1,2m-1}|_{\spp_{8m}(\mathcal{O})}\equiv1$. Then,
\[
f_{s}^{2m}(h)=\int_{F^{2m-1}}f_{s}^{1,2m-1}(m(\begin{pmatrix}1\\
 &  & 1\\
 & 1_{2m-1}\\
 &  &  & 1_{2m-1}
\end{pmatrix}\begin{pmatrix}1\\
 & 1_{2m-1} & M\\
 &  & 1\\
 &  &  & 1_{2m-1}
\end{pmatrix})h)dM
\]
holds by \cite[Lemma 22]{Cai:2019aa}. Therefore, we have

\begin{align*}
I & =\int_{F^{2m-1}}\int_{U_{0}^{3}}f_{s}^{1,2m-1}(u_{0}^{3}m(\begin{pmatrix}1\\
 & 1\\
 & M & 1_{2m-1}\\
 &  &  & 1_{2m-1}
\end{pmatrix}\diag(1_{2},A,1_{3m-1})))\psi_{U}^{3}(u_{0}^{3})du_{0}^{3}dM\\
 & =|\det A|\int_{F^{2m-1}}\int_{U_{0}^{3}}f_{s}^{1,2m-1}(u_{0}^{3}m(\diag(1_{2},A,1_{3m-1})\begin{pmatrix}1\\
 & 1\\
 & M & 1_{2m-1}\\
 &  &  & 1_{2m-1}
\end{pmatrix}))\psi_{U}^{3}(u_{0}^{3})du_{0}^{3}dM,
\end{align*}
where
\[
U_{0}^{3}=\left\{ \begin{pmatrix}1\\
 & 1\\
 &  & 1_{2m-1}\\
 &  &  & 1_{2m-1}\\
* & * & * & * & 1_{2m-1}\\
 & * &  & * &  & 1_{2m-1}\\
 & * & * & * &  &  & 1\\
 &  &  & * &  &  &  & 1
\end{pmatrix}\right\} 
\]
and $\psi_{U}^{3}(u_{0}^{3})=\psi_{F}^{-1}(\tr\begin{pmatrix}0_{m}\\
 & 1_{m-1}
\end{pmatrix}Y+\tr W\begin{pmatrix}0^{m-1}\\
1\\
0^{m-1}
\end{pmatrix})$ for 
\[
u_{0}^{3}=\begin{pmatrix}1\\
 & 1\\
 &  & 1_{2m-1}\\
 &  &  & 1_{2m-1}\\
X & P & Y & Q & 1_{2m-1}\\
 & W^{*} &  & Y^{*} &  & 1_{2m-1}\\
 & R & W & P^{*} &  &  & 1\\
 &  &  & X^{*} &  &  &  & 1
\end{pmatrix}.
\]

As above, we reduce the integration in $I$ by considering the right
action of some elements in $\spp_{8m}(\mathcal{O})$ like $\begin{pmatrix}1 &  & *\\
 & 1_{8m-2}\\
 &  & 1
\end{pmatrix}$. Here we present only the results.
\begin{itemize}
\item We have 
\[
I=|\det A|\int_{U_{0}^{3}}f_{s}^{1,2m-1}(u_{0}^{3}m(\diag(1_{2},A,1_{3m-1})))\psi_{U}^{3}(u_{0}^{3})du_{0}^{3}
\]
 by considering the action of $m(\begin{pmatrix}1 &  & (*,0^{m-1})\\
 & 1 &  & (0^{m-1},*)\\
 &  & 1_{2m-1}\\
 &  &  & 1_{2m-1}
\end{pmatrix}).$
\item We have 
\[
I=|\det A|\int_{U_{0}^{4}}f_{s}^{1,2m-1}(u_{0}^{4}m(\diag(1_{2},A,1_{3m-1})))\psi_{U}^{3}(u_{0}^{4})du_{0}^{4}
\]
 by considering the action of $\begin{pmatrix}1 &  &  &  &  &  & *\\
 & 1 &  &  &  &  &  & *\\
 &  & 1_{2m-1}\\
 &  &  & 1_{2m-1}\\
 &  &  &  & 1_{2m-1}\\
 &  &  &  &  & 1_{2m-1}\\
 &  &  &  &  &  & 1\\
 &  &  &  &  &  &  & 1
\end{pmatrix}$, where 
\[
U_{0}^{4}=\left\{ \begin{pmatrix}1\\
 & 1\\
 &  & 1_{2m-1}\\
 &  &  & 1_{2m-1}\\
* & * & * & * & 1_{2m-1}\\
 & * &  & * &  & 1_{2m-1}\\
 &  & * & * &  &  & 1\\
 &  &  & * &  &  &  & 1
\end{pmatrix}\right\} .
\]
\item We have 
\begin{align*}
I & =|\det A|\int_{U_{0}^{5}}f_{s}^{1,2m-1}(u_{0}^{5}m(\diag(1_{2},A,1_{3m-1})))\psi_{U}^{3}(u_{0}^{5})du_{0}^{5}\\
 & =|\det A|^{-2m+1}\int_{U_{0}^{5}}f_{s}^{1,2m-1}(m(\diag(1_{2},A,1_{3m-1}))u_{0}^{5})\psi_{U}^{3}(u_{0}^{5})du_{0}^{5}
\end{align*}
 by considering the action of $m(\begin{pmatrix}1 &  & (0^{m},*)\\
 & 1 &  & (*,0^{m})\\
 &  & 1_{2m-1}\\
 &  &  & 1_{2m-1}
\end{pmatrix})$, where 
\[
U_{0}^{5}=\left\{ \begin{pmatrix}1\\
 & 1\\
 &  & 1_{2m-1}\\
 &  &  & 1_{2m-1}\\
\begin{pmatrix}*\\
0^{m-1}
\end{pmatrix} & * & * & * & 1_{2m-1}\\
 & \begin{pmatrix}0^{m-1}\\
*
\end{pmatrix} &  & * &  & 1_{2m-1}\\
 &  & \begin{pmatrix}* & 0^{m-1}\end{pmatrix} & * &  &  & 1\\
 &  &  & \begin{pmatrix}0^{m-1} & *\end{pmatrix} &  &  &  & 1
\end{pmatrix}\right\} .
\]
\item We have 
\[
I=|\det A|^{-2m+1}\int_{F}\int_{U_{0}^{6}}f_{s}^{1,2m-1}(m(\diag(1_{2},A,1_{3m-1}))u_{0}^{6}v(x))\psi_{U}^{3}(u_{0}^{6})\psi_{F}^{-1}(x)du_{0}^{6}dx
\]
 by considering the action of $\begin{pmatrix}1 &  &  &  &  & *\\
 & 1\\
 &  & 1_{2m-1} &  &  &  &  & *\\
 &  &  & 1_{2m-1}\\
 &  &  &  & 1_{2m-1}\\
 &  &  &  &  & 1_{2m-1}\\
 &  &  &  &  &  & 1\\
 &  &  &  &  &  &  & 1
\end{pmatrix}$, where 
\[
v(x)=\begin{pmatrix}1\\
 & 1\\
 &  & 1_{2m-1}\\
 &  &  & 1_{2m-1}\\
-xe^{*} &  &  &  & 1_{2m-1}\\
 & xe^{*} &  &  &  & 1_{2m-1}\\
 &  & xe &  &  &  & 1\\
 &  &  & -xe &  &  &  & 1
\end{pmatrix},\ e=\begin{pmatrix}0^{m-1} & 1 & 0^{m-1}\end{pmatrix},
\]
and
\[
U_{0}^{6}=\left\{ \begin{pmatrix}1\\
 & 1\\
 &  & 1_{2m-1}\\
 &  &  & 1_{2m-1}\\
 & * & * & * & 1_{2m-1}\\
 &  &  & * &  & 1_{2m-1}\\
 &  &  & * &  &  & 1\\
 &  &  &  &  &  &  & 1
\end{pmatrix}\right\} .
\]
\end{itemize}
Put 
\[
F_{2}(x)=|\det A|^{-2m+1}\int_{U_{0}^{6}}f_{s}^{1,2m-1}(m(\diag(1_{2},A,1_{3m-1}))u_{0}^{6}v(x))\psi_{U}^{3}(u_{0}^{6})\psi_{F}^{-1}(x)du_{0}^{6}.
\]
 Then
\[
I=F_{2}(0)+\int_{|x|>1}F_{2}(x)dx.
\]
If $|x|>1$, by
\[
v(x)=m(\diag((x^{-1}1_{2},1_{m-1},x^{-1},1_{2m-2},x^{-1},1_{m-1})){}^{t}v(x)\begin{pmatrix} &  &  &  & e\\
 &  &  &  &  & -e\\
 &  & e' &  &  &  & -e^{*}\\
 &  &  & e' &  &  &  & e^{*}\\
-e^{*} &  &  &  & e'\\
 & e^{*} &  &  &  & e'\\
 &  & e\\
 &  &  & -e
\end{pmatrix}{}^{t}v(x^{-1}),
\]
where $e'=\diag(1_{m-1},0,1_{m-1})$, we have 
\begin{align*}
 & F_{2}(x)\\
= & |\det A|^{-2m+1}\int_{U_{0}^{6}}f_{s}^{1,2m-1}(m(\diag(1_{2},A,1_{3m-1}))u_{0}^{6}m(\diag(x^{-1}1_{2},1_{m-1},x^{-1},1_{2m-2},x^{-1},1_{m-1})){}^{t}v(x))\psi_{U}^{3}(u_{0}^{6})\psi_{F}^{-1}(x)du_{0}^{6}\\
= & |x|^{-4s-2r-2}\psi_{F}^{-1}(x)F_{2}(0).
\end{align*}
Note that
\begin{align*}
\int_{|x|>1}|x|^{-4s-2r-2}\psi_{F}^{-1}(x)dx & =\sum_{n\geq0}\int_{p^{-n-1}\mathcal{O}\setminus p^{-n}\mathcal{O}}p^{(n+1)(-4s-2r-2)}\psi_{F}^{-1}(x)dx\\
 & =-\int_{\mathcal{O}}p^{-4s-2r-2}\psi_{F}^{-1}(x)dx\\
 & =-p^{-4s-2r-2}.
\end{align*}
 Thus we obtain 
\begin{align*}
I & =|\det A|^{-2m+1}\zeta(4s+2r+2)^{-1}F_{2}(0)\\
 & =\zeta(4s+2r+2)^{-1}\int_{U_{0}^{6}}f_{s}^{1,2m-1}(u_{0}^{6}m(\diag(1_{2},A,1_{3m-1})))\psi_{U}^{3}(u_{0}^{6})\psi_{F}^{-1}(x)du_{0}^{6}.
\end{align*}
Put 
\[
F_{3}(P)=\zeta(4s+2r+2)^{-1}\int_{U_{0}^{7}}f_{s}^{1,2m-1}(u_{0}^{7}\begin{pmatrix}1\\
 & 1\\
 &  & 1_{2m-1}\\
 &  &  & 1_{2m-1}\\
 & P &  &  & 1_{2m-1}\\
 &  &  &  &  & 1_{2m-1}\\
 &  &  & P^{*} &  &  & 1\\
 &  &  &  &  &  &  & 1
\end{pmatrix}m(\diag(1_{2},A,1_{3m-1})))\psi_{U}^{3}(u_{0}^{7})du_{0}^{7},
\]
where
\[
U_{0}^{7}=\left\{ \begin{pmatrix}1\\
 & 1\\
 &  & 1_{2m-1}\\
 &  &  & 1_{2m-1}\\
 &  & * & * & 1_{2m-1}\\
 &  &  & * &  & 1_{2m-1}\\
 &  &  &  &  &  & 1\\
 &  &  &  &  &  &  & 1
\end{pmatrix}\right\} .
\]
Then, by considering the action of $\begin{pmatrix}1 &  &  &  & *\\
 & 1\\
 &  & 1_{2m-1}\\
 &  &  & 1_{2m-1} &  &  &  & *\\
 &  &  &  & 1_{2m-1}\\
 &  &  &  &  & 1_{2m-1}\\
 &  &  &  &  &  & 1\\
 &  &  &  &  &  &  & 1
\end{pmatrix}$, we have
\[
F_{3}(P)=\begin{cases}
F_{3}(0) & P\in\mathcal{O}^{2m-1},\\
0 & \text{otherwise}
\end{cases}
\]
and 
\[
I=\zeta(4s+2r+2)^{-1}\int_{V_{0}^{1}}f_{s+1/2}^{2m-1}(v_{0}^{1}m(\diag(A,1_{3m-1})))\psi_{V}^{1}(v_{0}^{1})dv_{0}^{1},
\]
where 
\[
V_{0}^{1}=\left\{ \begin{pmatrix}1_{2m-1}\\
 & 1_{2m-1}\\
* & * & 1_{2m-1}\\
 & * &  & 1_{2m-1}
\end{pmatrix}\right\} 
\]
 and $\psi_{V}^{1}(v_{0})=\psi_{F}^{-1}(\tr\begin{pmatrix}0_{m}\\
 & 1_{m-1}
\end{pmatrix}Y)$ for 
\[
v_{0}=\begin{pmatrix}1_{2m-1}\\
 & 1_{2m-1}\\
Y & Q & 1_{2m-1}\\
 & Y^{*} &  & 1_{2m-1}
\end{pmatrix}.
\]

Put $I'=\zeta(4s+2r+2)I$ and we do similar calculations. For $k=\begin{pmatrix} & 1_{m-1}\\
1\\
 &  & 1_{m-1}
\end{pmatrix},$ we have
\begin{align*}
I' & =\int_{V_{0}^{1}}f_{s+1/2}^{2m-1}(v_{0}^{1}m(\diag(A,1_{3m-1})\diag(k,k)))\psi_{V}^{1}(v_{0}^{1})dv_{0}^{1}\\
 & =\int_{V_{0}^{1}}f_{s+1/2}^{2m-1}(v_{0}^{1}m(\diag(1,A,1_{3m-2})))\psi_{V}^{2}(v_{0}^{1})dv_{0}^{1},
\end{align*}

where $\psi_{V}^{2}(v_{0}^{1})=\psi_{F}^{-1}(\tr\begin{pmatrix}0_{m-1}\\
 & 1_{m-1}
\end{pmatrix}Y)$ for 
\[
v_{0}^{1}=\begin{pmatrix}1\\
 & 1_{2m-2}\\
 &  & 1\\
 &  &  & 1_{2m-2}\\
X & Y & P & Q & 1_{2m-2}\\
Z & W & R & P^{*} &  & 1\\
 &  & W^{*} & Y^{*} &  &  & 1_{2m-2}\\
 &  & Z & X^{*} &  &  &  & 1
\end{pmatrix}.
\]
 Moreover, by considering the action of $\begin{pmatrix}1 &  & *\\
 & 1_{8m-6}\\
 &  & 1
\end{pmatrix}$, we have

\[
I'=\int_{V_{0}^{2}}f_{s+1/2}^{2m-1}(v_{0}^{2}m(\diag(1,A,1_{3m-2})))\psi_{V}^{2}(v_{0}^{2})dv_{0}^{2},
\]
where 
\[
V_{0}^{2}=\left\{ \begin{pmatrix}1\\
 & 1_{2m-2}\\
 &  & 1\\
 &  &  & 1_{2m-2}\\
* & * & * & * & 1_{2m-2}\\
 & * & * & * &  & 1\\
 &  & * & * &  &  & 1_{2m-2}\\
 &  &  & * &  &  &  & 1
\end{pmatrix}\right\} .
\]

Let $f_{s+1/2}^{1,2m-2}$ be the section of $(\mathcal{W}(\sigma)\abs^{-m+1}\times\mathcal{W}(\sigma_{2m-2})\abs^{1/2})\abs^{s+1/2}\rtimes1_{\spp_{0}}$
satisfying $f_{s+1/2}^{1,2m-2}|_{\spp_{8m-4}(\mathcal{O})}\equiv1$.
Then, by
\[
f_{s+1/2}^{2m-1}(h)=\int_{F^{2m-2}}f_{s+1/2}^{1,2m-2}(m(\begin{pmatrix}1\\
 &  & 1\\
 & 1_{2m-2}\\
 &  &  & 1_{2m-2}
\end{pmatrix}\begin{pmatrix}1\\
 & 1_{2m-2} & M\\
 &  & 1\\
 &  &  & 1_{2m-2}
\end{pmatrix})h)dM,
\]
we have
\[
I'=|\det A|\int_{F^{2m-2}}\int_{V_{0}^{3}}f_{s+1/2}^{1,2m-2}(v_{0}^{3}m(\diag(1_{2},A,1_{3m-3})\begin{pmatrix}1\\
 & 1\\
 & M & 1_{2m-2}\\
 &  &  & 1_{2m-2}
\end{pmatrix}))\psi_{V}^{3}(v_{0}^{3})dv_{0}^{3}dM,
\]
where
\[
V_{0}^{3}=\left\{ \begin{pmatrix}1\\
 & 1\\
 &  & 1_{2m-2}\\
 &  &  & 1_{2m-2}\\
* & * & * & * & 1_{2m-2}\\
 & * &  & * &  & 1_{2m-2}\\
 & * & * & * &  &  & 1\\
 &  &  & * &  &  &  & 1
\end{pmatrix}\right\} 
\]
and $\psi_{V}^{3}(v_{0}^{3})=\psi_{F}^{-1}(\tr\begin{pmatrix}0_{m-1}\\
 & 1_{m-1}
\end{pmatrix}Y)$ for 
\[
v_{0}^{3}=\begin{pmatrix}1\\
 & 1\\
 &  & 1_{2m-2}\\
 &  &  & 1_{2m-2}\\
X & P & Y & Q & 1_{2m-2}\\
 & W^{*} &  & Y^{*} &  & 1_{2m-2}\\
 & R & W & P^{*} &  &  & 1\\
 &  &  & X^{*} &  &  &  & 1
\end{pmatrix}.
\]
 Then, the calculations continue as follows.
\begin{itemize}
\item We have 
\[
I'=|\det A|\int_{V_{0}^{3}}f_{s+1/2}^{1,2m-2}(v_{0}^{3}m(\diag(1_{2},A,1_{3m-3})))\psi_{V}^{3}(v_{0}^{3})dv_{0}^{3}
\]
 by considering the action of $m(\begin{pmatrix}1 &  & (*,0^{m-1})\\
 & 1 &  & (0^{m-1},*)\\
 &  & 1_{2m-2}\\
 &  &  & 1_{2m-2}
\end{pmatrix})$,
\item We have
\[
I'=|\det A|\int_{V_{0}^{4}}f_{s+1/2}^{1,2m-2}(v_{0}^{4}m(\diag(1_{2},A,1_{3m-3})))\psi_{V}^{3}(v_{0}^{4})dv_{0}^{4}
\]
 by considering the action of $\begin{pmatrix}1 &  &  &  &  &  & *\\
 & 1 &  &  &  &  &  & *\\
 &  & 1_{2m-2}\\
 &  &  & 1_{2m-2}\\
 &  &  &  & 1_{2m-2}\\
 &  &  &  &  & 1_{2m-2}\\
 &  &  &  &  &  & 1\\
 &  &  &  &  &  &  & 1
\end{pmatrix}$, where 
\[
V_{0}^{4}=\left\{ \begin{pmatrix}1\\
 & 1\\
 &  & 1_{2m-2}\\
 &  &  & 1_{2m-2}\\
* & * & * & * & 1_{2m-2}\\
 & * &  & * &  & 1_{2m-2}\\
 &  & * & * &  &  & 1\\
 &  &  & * &  &  &  & 1
\end{pmatrix}\right\} ,
\]
\item We have 
\[
I'=|\det A|^{-2m+2}\int_{V_{0}^{5}}f_{s+1/2}^{1,2m-2}(m(\diag(1_{2},A,1_{3m-3}))v_{0}^{5})\psi_{V}^{3}(v_{0}^{5})dv_{0}^{5}
\]
 by considering the action of $m(\begin{pmatrix}1 &  & (0^{m-1},*)\\
 & 1 &  & (*,0^{m-1})\\
 &  & 1_{2m-2}\\
 &  &  & 1_{2m-2}
\end{pmatrix})$, where 
\[
U_{0}^{5}=\left\{ \begin{pmatrix}1\\
 & 1\\
 &  & 1_{2m-2}\\
 &  &  & 1_{2m-2}\\
\begin{pmatrix}*\\
0^{m-1}
\end{pmatrix} & * & * & * & 1_{2m-2}\\
 & \begin{pmatrix}0^{m-1}\\
*
\end{pmatrix} &  & * &  & 1_{2m-2}\\
 &  & \begin{pmatrix}* & 0^{m-1}\end{pmatrix} & * &  &  & 1\\
 &  &  & \begin{pmatrix}0^{m-1} & *\end{pmatrix} &  &  &  & 1
\end{pmatrix}\right\} .
\]
\item We have
\begin{align*}
I' & =|\det A|^{-2m+1}\int_{V_{0}^{6}}f_{s+1/2}^{1,2m-2}(m(\diag(1_{2},A,1_{3m-3}))v_{0}^{6})\psi_{V}^{3}(v_{0}^{6})dv_{0}^{6}\\
 & =\int_{V_{0}^{6}}f_{s+1/2}^{1,2m-2}(v_{0}^{6}m(\diag(1_{2},A,1_{3m-3})))\psi_{V}^{3}(v_{0}^{6})dv_{0}^{6}
\end{align*}
 by considering the action of $\begin{pmatrix}1 &  &  &  &  & *\\
 & 1\\
 &  & 1_{2m-2} &  &  &  &  & *\\
 &  &  & 1_{2m-2}\\
 &  &  &  & 1_{2m-2}\\
 &  &  &  &  & 1_{2m-2}\\
 &  &  &  &  &  & 1\\
 &  &  &  &  &  &  & 1
\end{pmatrix}$, where 
\[
V_{0}^{6}=\left\{ \begin{pmatrix}1\\
 & 1\\
 &  & 1_{2m-1}\\
 &  &  & 1_{2m-1}\\
 & * & * & * & 1_{2m-1}\\
 &  &  & * &  & 1_{2m-1}\\
 &  &  & * &  &  & 1\\
 &  &  &  &  &  &  & 1
\end{pmatrix}\right\} .
\]
\end{itemize}
Then, by considering the action of $\begin{pmatrix}1 &  &  &  & *\\
 & 1\\
 &  & 1_{2m-1}\\
 &  &  & 1_{2m-1} &  &  &  & *\\
 &  &  &  & 1_{2m-1}\\
 &  &  &  &  & 1_{2m-1}\\
 &  &  &  &  &  & 1\\
 &  &  &  &  &  &  & 1
\end{pmatrix}$, we have 
\[
I'=\int_{W_{0}^{1}}f_{s+1}^{2m-2}(w_{0}^{1}m(\diag(A,1_{3m-3})))\psi_{W}^{1}(w_{0}^{1})dw_{0}^{1},
\]
where 
\[
W_{0}^{1}=\left\{ \begin{pmatrix}1_{2m-2}\\
 & 1_{2m-2}\\
* & * & 1_{2m-2}\\
 & * &  & 1_{2m-2}
\end{pmatrix}\right\} 
\]
 and $\psi_{W}^{1}(W_{0}^{1})=\psi_{F}^{-1}(\tr X\begin{pmatrix}0_{m-1}\\
 & 1_{m-1}
\end{pmatrix})$ for 
\[
w_{0}^{1}=\begin{pmatrix}1_{2m-2}\\
 & 1_{2m-2}\\
X & * & 1_{2m-2}\\
 & * &  & 1_{2m-2}
\end{pmatrix}.
\]
Thus we have

\[
I=\zeta(4s+2r+2)^{-1}F_{s+1,m-1}(m(A)).
\]
\end{proof}
Next, we rewrite $\Phi$ using $F_{s,2n+2r}$. The key is the following
result in \cite{zbMATH07485546} (local double descent). 
\begin{prop}[{\cite[Theorem 12.3]{zbMATH07485546}}]
 \label{prop:-For-,}For $m\in\zz_{>0}$, put 
\[
\rho_{m}^{0}=\chi(\det{}_{3m})\times\chi(\det{}_{m})\rtimes1_{\spp_{0}}.
\]
Then, we have

\[
J_{\psi_{U}^{-1}}(\rho_{m}^{0})\simeq(\chi(\det{}_{m})\rtimes1_{\spp_{0}})\boxtimes(\chi(\det{}_{m})\rtimes1_{\spp_{0}})
\]
as a representation of $i(\spp_{2m}(F),\spp_{2m}(F))\simeq\spp_{2m}(F)\times\spp_{2m}(F)$.
(Note: $\chi(\det{}_{m})\rtimes1_{\spp_{0}}\simeq(\chi(\det{}_{m})\rtimes1_{\spp_{0}})^{\iota}\simeq(\chi(\det{}_{m})\rtimes1_{\spp_{0}})^{\vee}$.)
\end{prop}

\begin{rem}
Denote by $\rho_{m}$ the unramified constituent of $\rho_{m}^{0}$.
Assume that $m$ is even. Then, we can see that $\rho_{m}\subset\sigma_{2m}\abs^{-m/2}\rtimes1_{\spp_{0}}$
as follows. 

First, $\chi(\det{}_{3m})\subset\chi(\det{}_{2m})\abs^{-m/2}\times\chi(\det{}_{m})\abs^{m}$
induces

\[
\rho_{m}(\subset\rho_{m}^{0})\subset\chi(\det{}_{2m})\abs^{-m/2}\times\chi(\det{}_{m})\abs^{m}\times\chi(\det{}_{m})\rtimes1_{\spp_{0}}=:\rho_{m}^{1}.
\]
Next, $\chi(\det{}_{m})\abs^{m/2}\times\chi(\det{}_{m})\abs^{-m/2}\surj\chi(\det{}_{2m})$
induces

\[
\rho_{m}^{1}\surj\chi(\det{}_{2m})\abs^{-m/2}\times\chi(\det{}_{2m})\abs^{m/2}\rtimes1_{\spp_{0}}=:\rho_{m}^{2}.
\]
Note that this map send a nonzero unramified vector in $\rho_{m}$
to a nonzero unramified vector in $\rho_{m}^{2}$. Then, since $(\chi(\det{}_{2m})\abs^{m/2})^{\pm1}\rtimes1_{\spp_{0}}$
are irreducible (by \cite{MR1194967}) and $\sigma_{2m}\simeq\chi(\det{}_{2m})\times\chi^{-1}(\det{}_{2m})$,
we have $\chi^{-1}(\det{}_{2m})\abs^{-m/2}\rtimes1_{\spp_{0}}\simeq\chi(\det{}_{2m})\abs^{m/2}\rtimes1_{\spp_{0}}$
and it induces

\[
\rho_{m}^{2}\overset{\sim}{\map}\sigma_{2m}\abs^{-m/2}\rtimes1_{\spp_{0}}.
\]
And finally, $\rho_{m}$ is mapped onto a nonzero subspace of $\sigma_{2m}\abs^{-m/2}\rtimes1_{\spp_{0}}$
by the composition of above maps.
\end{rem}

\begin{cor}
\label{cor:mc-1}Assume that $m$ is even. Let $\Theta$ be the bi-$\spp_{2m}(\mathcal{O})$-invariant
matrix coefficient of $\chi(\det{}_{m})\rtimes1_{\spp_{0}}$ such
that $\Theta(1)=1$. Then, 

\[
\Delta_{\spp_{2m}}^{-1}\Theta=F_{-m/2,m}
\]
holds.
\end{cor}

\begin{proof}
We regard $\rho_{m}\subset\mathcal{W}(\sigma_{2m})\abs^{-m/2}\rtimes1_{\spp_{0}}$.
For now, we assume that the following map

\begin{align*}
L:\rho_{m} & \map\cc\\
f_{-m/2} & \mapsto\mathcal{F}(f_{-m/2})(1,1)
\end{align*}
is well-defined. Then, since $L$ goes through $J_{\psi_{U}^{-1}}(\rho_{m})$
and $L$ is $\Delta(\spp_{2m}(F))$-invariant (see Remark \ref{rem:We-note-some}),
$\mathcal{F}(f_{-m/2})(1,g^{\iota})$ is the sum of some matrix coefficients
of $\chi(\det{}_{m})\rtimes1_{\spp_{0}}$ for any $f_{-m/2}\in\rho_{m}$
by 
\[
J_{\psi_{U}^{-1}}(\rho_{m})\subset J_{\psi_{U}^{-1}}(\rho_{m}^{0})\simeq(\chi(\det{}_{m})\rtimes1_{\spp_{0}})^{\vee}\boxtimes(\chi(\det{}_{m})\rtimes1_{\spp_{0}})^{\iota}
\]
(Proposition \ref{prop:-For-,}). Thus, since $F_{-m/2,m}$ is bi-$\spp_{2m}(\mathcal{O})$-invariant,
we have 
\[
F_{-m/2,m}=F_{-m/2,m}(1)\Theta=\Delta_{\spp_{2m}}^{-1}\Theta
\]
 by Proposition \ref{prop:Assume-that-}.

Now we show that $L$ is well-defined. Define $L_{s}$ by $L_{s}(f_{s})=\mathcal{F}(f_{s})(1,1)$
for any standard section $f_{s}$ of $\mathcal{W}(\sigma_{2m})\abs^{s}\rtimes1_{\spp_{0}}$
satisfying $f_{-m/2}\in\rho_{m}.$ Consider the Laurent expansion

\[
L_{s}=\sum_{M\geq-M_{0}}(s+m/2)^{M}L^{M}\circ J
\]
of $L_{s}$ at $s=-m/2,$ where $M_{0}\in\zz_{\geq0},$ $J$ is the
canonical surjection from $\rho_{m}$ to $J_{\psi_{U}^{-1}}(\rho_{m})$,
and 
\[
L^{M}\in\mathrm{Hom}_{\Delta(\spp_{2m}(F))}(J_{\psi_{U}^{-1}}(\rho_{m}),\cc)\subset\mathrm{Hom}_{\Delta(\spp_{2m}(F))}(J_{\psi_{U}^{-1}}(\rho_{m}^{0}),\cc).
\]
 Since $m$ is even, $\chi(\det{}_{m})\rtimes1_{\spp_{0}}$ is irreducible
by \cite{MR1194967}. In particular, 
\[
\mathrm{Hom}_{\Delta(\spp_{2m}(F))}(J_{\psi_{U}^{-1}}(\rho_{m}^{0}),\cc)\simeq\mathrm{Hom}_{\spp_{2m}(F)}((\chi(\det{}_{m})\rtimes1_{\spp_{0}})^{\vee}\otimes(\chi(\det{}_{m})\rtimes1_{\spp_{0}})^{\iota},\cc)
\]
 is 1-dimensional. Fix a nonzero element $\mathfrak{L}$ in $\mathrm{Hom}_{\Delta(\spp_{2m}(F))}(J_{\psi_{U}^{-1}}(\rho_{m}^{0}),\cc)$
and put $L^{M}=c_{M}\mathfrak{L},$ where $c_{M}\in\cc$. Then, since
$F_{s,m}$ is well-defined and nonzero at $s=-m/2$, we have $c_{M}\mathfrak{L}(J(f_{-m/2}^{2m}))=0$
for $M>0$ and $c_{0}\mathfrak{L}(J(f_{-m/2}^{2m}))\neq0$. Thus $c_{M}=0$
for $M>0$ and $L=L_{-m/2}=c_{0}\mathfrak{L}$ is well-defined.
\end{proof}
And finally, we complete the proof of Theorem \ref{thm:Assume-that-}.
\begin{proof}[Proof of Theorem \ref{thm:Assume-that-} for $v=p<\infty$]
 By Proposition \ref{prop:Assume-that-} and Corollary \ref{cor:mc-1},
we have 
\[
\Phi(g)=\Delta_{\spp_{4n+4r}}F_{-n-r,2n+2r}(g)=(\Delta_{\spp_{4n+4r}}/\Delta_{\spp_{4n+2r}})F_{n,r}(g)
\]
 for $g\in\spp_{2r}(F).$ Thus we have

\begin{align*}
I(\Phi,\Psi) & =(\Delta_{\spp_{4n+4r}}/\Delta_{\spp_{4n+2r}})\int_{\spp_{2r}(F)}F_{n,r}(g)\overline{\Psi(g)}dg\\
 & =\frac{\Delta_{\spp_{4n+4r}}}{\Delta_{\spp_{2r}}\Delta_{\spp_{4n+2r}}}\frac{L(n+1/2,\pi\times\sigma)}{L(n+r+1/2,\sigma)\prod_{i=1}^{r}L(2n+2i-1,\sigma,\mathrm{Ad})\zeta(2n+2i)}.
\end{align*}
\end{proof}
\bibliographystyle{plain}
\bibliography{cloudbib}

\end{document}